\documentclass{amsart}
\usepackage[utf8]{inputenc}

\usepackage{graphics}
\usepackage{thmtools}
\usepackage{wasysym}

\usepackage{amsthm}
\usepackage{amsbsy,amsmath,amssymb,amscd,amsfonts}
\usepackage[pagebackref=false]{hyperref}

\usepackage{graphicx,float,latexsym,color}
\usepackage[graphicx]{realboxes}
\usepackage[font={scriptsize,it}]{caption}
\usepackage{subcaption}

\usepackage{makecell}

\usepackage[dvipsnames]{xcolor}

\newtheorem{theorem}{Theorem}
\newtheorem*{theorem*}{Theorem}

\newtheorem{proposition}{Proposition}
\newtheorem{conjecture}{Conjecture}
\newtheorem{corollary}{Corollary}
\newtheorem{lemma}{Lemma}
\theoremstyle{remark}
\newtheorem{remark}{Remark}
\theoremstyle{definition}

\hypersetup{
    pdftoolbar=true,        
    pdfmenubar=true,        
    pdffitwindow=false,     
    pdfstartview={FitH},    
    colorlinks=true,       
    linkcolor=OliveGreen,          
    citecolor=blue,        
    filecolor=black,      
    urlcolor=red           
}

\usepackage{lineno}

\usepackage{comment}
\usepackage{microtype}
\usepackage{footnote}

\newcommand{\E}{\mathcal{E}}
\newcommand{\T}{\mathcal{T}}
\renewcommand{\L}{\mathcal{L}}
\newcommand{\X}{\mathcal{X}_\rho}
\newcommand{\Vs}{{V_1}{V_2}}
\newcommand{\torp}[2]{\texorpdfstring{#1}{#2}}

\title[Invariants of Centers Ellipse-Inscribed Triangles]{Intriguing Invariants of Centers\\of Ellipse-Inscribed Triangles}
\author{Mark Helman}
\thanks{M. Helman, Dept. of Mathematics, Rice University, Houston, TX. \texttt{markhelman@hotmail.com}}
\author{Ronaldo Garcia} 
\thanks{R. Garcia, Math. \& Stats. Inst., Federal Univ. of Goiás, Goiânia, Brazil. \texttt{ragarcia@ufg.br}}
\author{Dan Reznik$^*$}
\thanks{D. Reznik$^*$, Data Science Consulting Ltd., Rio de Janeiro, Brazil. \texttt{dreznik@gmail.com}}
\date{September, 2020}

\begin{document}

\maketitle

\begin{abstract}
We describe intriguing properties of a 1d family of triangles: two vertices are pinned to the boundary of an ellipse while a third one sweeps it. We prove that: (i) if a triangle center is a fixed affine combination of barycenter and orthocenter, its locus is an ellipse; (ii) over the family of said affine combinations, the centers of said loci sweep a line; (iii) over the family of parallel fixed vertices, said loci rigidly translate along a second line. Additionally, we study invariants of the envelope of elliptic loci over combinations of two fixed vertices on the ellipse.

\end{abstract}

\section{Introduction}
\label{sec:intro}
Consider the 1d family of triangles $\T(t) = V_1 V_2 P(t)$ where two vertices $V_1,V_2$ are fixed to the boundary of an ellipse while a third one $P(t)$ sweeps it. A known property \cite{dykstra2006-loci} is that over this family, the locus of classic centers such as the bary-, circum-, and orthocenter (denoted $X_k,k=2,3,4$, after \cite{etc}) are all ellipses; see  Figure~\ref{fig:locus-x234}. We further generalize these results and show that:

\begin{itemize}
    \item If a triangle center $X_k$ is a fixed affine combination of bary- $X_2$ and orthocenter $X_4$, its locus is an ellipse (co-discovered with A. Akopyan \cite{akopyan2020-private-loci}).
    \item Over the family of said affine combinations, centers of said loci sweep a first line. 
    \item Over parallel $\Vs$, the elliptic locus of $X_k$ is a family of ellipses which rigidly translate along a second line passing through $O$, the center of $\E$.
\end{itemize}

Additionally, we study the family of elliptic loci of $X_k$ for fixed $V_1$ and over all $V_2$ on $\E$. We show that  (i) the locus of their centers is an ellipse axis-aligned with $\E$, and that (ii) the external envelope to the locus family is invariant over $V_1$. In particular, the external envelope of $X_4$ (resp. $X_2$'s) loci is an affine image of Pascal's Limaçon (a 1/3-sized copy of $\E$).

\begin{figure}
    \centering
    \includegraphics[width=.66\textwidth]{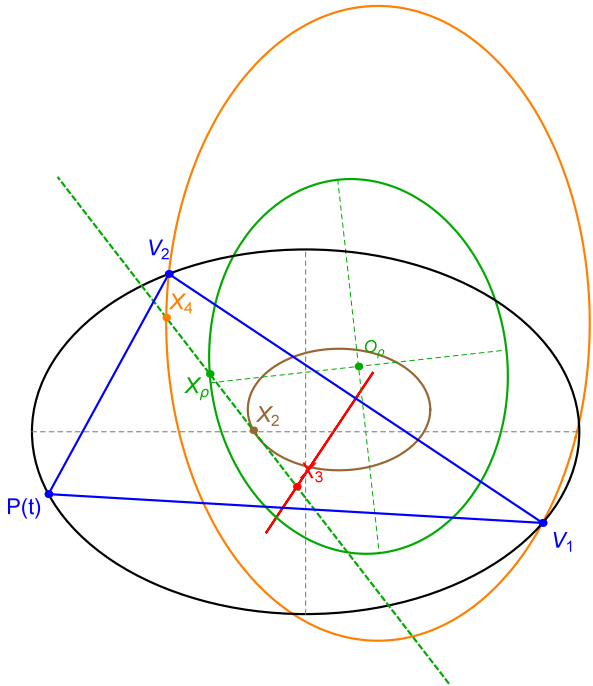}
    \caption{A triangle $V_1 V_2 P(t)$ (blue) is inscribed in an ellipse $\E$ (black) with semi-axes $a,b$. While $V_1,V_2$ are fixed, $P(t)$ executes one revolution along the boundary. Consider triangle centers $X_k,k=2,3,4$ on the Euler line $\L_e$ (dashed green). For any choice of $\Vs$, the locus of (i) $X_2$ is an axis-aligned ellipse (brown) with semi-axes $a/3,b/3$; (ii) $X_3$ is a segment (red) contained in the perp. bisector of $\Vs$; (iii) $X_4$ is an axis-aligned ellipse passing through $V_1,V_2$, with aspect ratio $b/a$. An intermediate point $X_\rho$ ($\rho=1/2$) is shown on $\L_e$ as well as its elliptic locus (green) centered on $O_\rho$. Though an ellipse for any $\rho$, this locus is in general neither axis-aligned (notice its slanted axes, dashed green), nor similar to $\E$. \href{https://youtu.be/zjiNgfndBWg}{Video}}
    \label{fig:locus-x234}
\end{figure}

\subsection*{Related Work}

In \cite{monroe2019-ortho}, an explicit derivation is given for the elliptic locus of the orthocenter for the same triangle family studied in this article. Dykstra et al. \cite{dykstra2006-loci} derived equations for loci of triangle centers under triangle families with two vertices affixed to special points of the ellipse, including the foci. Loci of triangle centers have been studied for alternative triangle families including: (i) Poristic triangles (those with fixed incircle and circumcircle) \cite{Gallatly1913,Weaver1924,Weaver1933,Murnaghan1925}. Odehnal describes pointwise, circular, and elliptic loci for dozens of triangle centers in the poristic family \cite{odehnal2011-poristic}. The poristic family is related via a similarity transform to another ``famous'' family: 3-periodics in the elliptic billiard \cite{garcia2020-poristic}; (ii) triangles with common incircle and centroid. Pamfilos has shown their vertices lie on a conic \cite{Pamfilos2011}; (iii)
Triangles with sides tangent to a circle \cite{Nikolina-families2012}; (iv) Triangles associated with two lines and a point not on them \cite{Sliepcevic2013}, etc. (v) Stanev has desrives the (ellptic) locus for the centroid of ellipse-inscribed equilaterals \cite{stanev2019-equilocus}. 

Some examples of polygon families include (i) rectangles inscribed in smooth curves \cite{schwartz2018-rectangles}, (ii) Poncelet polygons and the locus of their centroids \cite{schwartz2016-com} or their circumcenter \cite{caliz2021-poncelet}. 

We've studied loci of 3-periodics in the elliptic billiard, having found experimentally that the locus of the incenter is an ellipse \cite{reznik2011-incenter}. This was subsequently proved \cite{olga14}. Appearing thereafter were proofs for the elliptic locus of the barycenter \cite{sergei07_grid} and circumcenter \cite{fierobe2021-x3,garcia2019-incenter}. More recently, with the help of a computer algebra system (CAS), we showed that 29 out of the first 100 entries in \cite{etc} are ellipses over billiard 3-periodics, though what determines ellipticity is still not understood  \cite{garcia2020-ellipses}. We've also studied properties (e.g., area invariance) of the negative pedal curve (NPC) of the ellipse with respect to a point on its boundary \cite{garcia2020-steiner} as well as its pedal-like derivatives \cite{reznik2020-pedals}. We have also studied the locus of Brocard points over circle- and ellipse-inscribe triangle families \cite{garcia2020-brocard}.

The envelope of Euler lines in triangles with two fixed vertices at the foci of an ellipse was studied in \cite{macqueen1946}. Note: we study loci of this family in Appendix~\ref{app:focus-mounted}.

\subsection*{Article Structure} Section~\ref{sec:defns} defines basic concepts used in the article. Section~\ref{sec:main-results} contains our main
results. Section~\ref{sec:transl} analyzes loci over families of parallel $\Vs$. Section~\ref{sec:loci-env} analyzes the envelope of loci with fixed $V_1$ and varying $V_2$. In the conclusion, Section~\ref{sec:conclusion}, a table is included with videos mentioned throughout the paper. A set of appendices is included: Appendix~\ref{app:focus-mounted} overviews triangle centers whose loci are ellipses for a closely-related triangle family: $V_1 V_2$ are fixed on the foci of $\E$. Appendix~\ref{app:arho-brho-ratio} includes longer calculations used in derivations of key properties of elliptic loci. Finally, Appendix~\ref{app:symbols} contains a quick-reference to most symbols used herein. 

\section{Definitions}
\label{sec:defns}
Let $(a,b)$ be the semi-axes of ellipse $\E$ centered on $O$, $a{\geq}b$. Let $U(s)=[a\cos{s},b\sin{s}]$ be a generic point on its boundary. Let $t_1,t_2$ be two constants, and define $V_1=U(t_1)$, $V_2=U(t_2)$. Consider the triangle family $\T(t)={\Vs}{P(t)}$, where $P(t)=U(t)$, $t\in[0,2\pi)$; see Figure~\ref{fig:locus-x234}.

Given a $\T(t)$, Let $\X$ be a fixed affine combination of barycenter $X_2$ and orthocenter $X_4$ (both on Euler Line $\L_e$), i.e., $\X=(1 - \rho) X_2 + \rho X_4$; Figure~\ref{fig:euler} depicts the relative locations of the first 16 centers on \cite{etc} with fixed $\rho$, whereas  Table~\ref{tab:xrho} specifies their respective $\rho$.

\begin{figure}
    \centering
    \includegraphics[width=\textwidth]{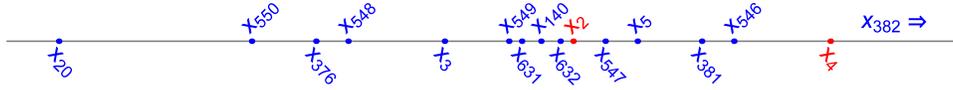}
    \caption{Relative locations of the first 16 triangle centers on \cite{etc} which are both on the Euler line and which are fixed affine combinations of $X_2$ and $X_4$.}
    \label{fig:euler}
\end{figure}

\setlength{\tabcolsep}{2pt}
\begin{table}[H]
{\small
\begin{tabular}{|c||c|c|c|c|c|c|c|c|c|c|c|c|c|c|c|c|}
 \hline
 $X_k$ & 20 & 550 & 376 & 548 & 3 & 549 & 631 & 140 & 632 & \textbf{2} & 547 & 5 & 381 & 546 & \textbf{4} & 382 \\
\hline
 $\rho$ & -2 & -1.25 & -1 & -0.875 & -0.5 & -0.25 & -0.2 & -0.125 & -0.05 & \textbf{0} & 0.125 & 0.25 & 0.5 & 0.625 & \textbf{1} & 2.5 \\
 \hline
\end{tabular}
}
\caption{First 16 points on \cite{etc} on Euler line and at fixed $\rho$.}
    \label{tab:xrho}
\end{table}

Currently, upwards of 38k+ {\em triangle centers} are catalogued in \cite{etc}. These are points on the plane of a triangle whose trilinear (or barycentric) coordinates are functions of side lengths and angles. Said functions must satisfy homogeneity, bisymmetry, and cyclicity   \cite[Triangle Center]{mw}. Examples include such classic centers as the incenter $X_1$, barycenter $X_2$, circumcenter $X_3$, orthocenter 

\begin{remark}
Since $\X$ is an affine combination of two triangle centers, its barycentrics add in a similar manner and therefore $\X$ is also a triangle center \cite{moses2020-private-barys}.
\end{remark}

\noindent Finally, let $c^2=a^2-b^2$, $d^2=a^2+b^2$, and $z=\cos(t_1+t_2)$.

\section{Main Results}
\label{sec:main-results}
\noindent Referring to Figure~\ref{fig:locus-x234}:

\begin{remark}
Since $X_2=[V_1+V_2+P(t)]/3$, the locus of $X_2$ is an ellipse $\E_2$ centered on $O_2=(V_1+V_2)/3$, axis-parallel with $\E$, and with semi-axes $(a_2,b_2)=(a/3,b/3)$.
\label{rem:x2}
\end{remark}

\begin{proposition}
For any $V_1$, $V_2$, the locus of $X_4$ is an ellipse $\E_4$ of semi-axes $a_4,b_4$, centered on $O_4$, axis-parallel with $\E$, and with aspect ratio $b/a$. These are given by:
\begin{align*}
    O_4&=\frac{c^2}{2}\left[\frac{\cos{t_1}+\cos{t_2}}{a}, \frac{\sin{t_1}+\sin{t_2}}{b}\right]\\
    (a_4,b_4)&=(w/a,w/b),\;\;\mbox{where: }
    w=\frac{\sqrt{2}}{2}\left(\sqrt{a^4+b^4 + (b^4-a^4) z  }\right) 
\end{align*}

\end{proposition}

\begin{proof}
Note that the pencils of sidelines $V_1P(t)$ and $V_2P(t)$ are invariant under projective transformations, and are respectively congruent to that of the altitudes $V_2X_4$ and $V_1X_4$. Thus, Steiner's generation of conics \cite{coxeter93} shows immediately that the locus of $X_4$ is an ellipse similar to E, but rotated $90^\circ$.

Using a CAS for simplification obtain the following parametric for $\E_4$:

\begin{align*}
\E_4&:[x_4(t),y_4(t)],t\in[0,2\pi), \mbox{where:}\\
x_4(t)&=\frac{d^2 (\cos t_1+\cos t_2+\cos t)+c^2 \cos (t_1+t_2+t)}{2 a}\\
y_4(t)&=\frac{d^2 (\sin t_1 +\sin t_2+\sin t)+c^2 \sin (t_1+t_2+t)}{2 b}
\end{align*} 
\end{proof}

\begin{theorem}
For any $V_1$, $V_2$, and any $\rho$, the locus of $\X$ is an ellipse (in general not axis-parallel with $\E$) with parametric:

\begin{align*}
\E_{\rho}&:[x_{\rho}(t),y_{\rho}(t)],t\in[0,2\pi), \mbox{where:}\\
x_{\rho}=&  \frac{(a^2(\rho+2)+3b^2\rho)(\cos t_1+\cos t_2+\cos t)}{6a} -\frac{c^2\rho\cos(t+t_1 +t_2)  }{2a} \\
y_{\rho}=& \frac{(a^2(\rho+2)+3b^2\rho)(\sin t_1+\sin t_2+\sin t)}{6b} -\frac{c^2\rho\sin(t+t_1 +t_2)  }{2b}
\end{align*}

\noindent and center $O_\rho$ at:

\[ O_\rho=\left[\frac{\left(a^2 (\rho +2)+3 b^2 \rho \right) (\cos{t_1}+\cos{t_2})}{6 a},\frac{\left(3 a^2 \rho +b^2 (\rho +2)\right) (\sin{t_1}+\sin{t_2})}{6 b}\right]
\]
 
\label{thm:ellipse}

\end{theorem}

\begin{proof}
CAS manipulation of $\X=(1-\rho)X_2 +\rho X_4$.
\end{proof}

\noindent in Appendix~\ref{app:fixed-rho} we've included all 226 triangle centers on \cite{etc} which are fixed affine combinations of $X_2,X_4$. 

As before, let $U(s)$ denote $[a \cos{s},b\sin{s}]$ on $\E$.

\begin{lemma}
Lines $U(t_1)U(t_2)$ are parallel for all $t_1+t_2=t_0$, where $t_0$ is some constant.
\label{lem:v1v2-parallel}
\end{lemma}

\begin{proof}
The slope of the line through points $U(t_1)=[a \cos(t_1),b \sin(t_1)]$ and $U(t_2)=[a \cos(t_2),b \sin(t_2)]$ is given by \[
\frac{b(\sin{t_1}-\sin{t_2})}{a(\cos{t_1}-\cos{t_2})}=\frac{2b\cos(\frac{t_1+t_2}{2})\sin(\frac{t_1-t_2}{2})}{-2a\sin(\frac{t_1+t_2}{2})\sin(\frac{t_1-t_2}{2})}=-\frac{b}{a}\cot\left(\frac{t_0}{2}\right)
\]
which only depends on $t_0$.
\end{proof}

\begin{remark}
Lemma~\ref{lem:v1v2-parallel}  implies that if $\Vs$ are vertical (resp. horizontal), then $t_1+t_2=2k\pi$ (resp. $t_1+t_2=(2k+1)\pi$), where $k$ is an integer.
\end{remark}

\begin{corollary}
For $\Vs$ vertical or horizontal, for all $\rho$, the locus of $\X$ is axis-parallel with $\E$.
\end{corollary}

\begin{proof}
Follows from the implicit equation for $\X$ (Appendix \ref{app:arho-brho-ratio}). Specifically, the coefficient $a_{11}$ of $x y$ vanishes whenever $t_1+t_2$ is an integer multiple of $\pi$.
\end{proof}

Let $(a_\rho,b_\rho)$ denote the semi-axes of the locus of $\X$.

\begin{proposition}
The ratio $a_\rho/b_\rho$ is invariant over the family of parallel $\Vs$.
\label{prop:ratio}
\end{proposition}

\begin{proof}
Rewrite the ellipse in Theorem~\ref{thm:ellipse} implicitly, and using a CAS obtain the ratio of eigenvalues of its Hessian, yielding the expression for $a_\rho/b_\rho$ in Appendix~\ref{app:arho-brho-ratio}. Notice its non-constant terms only depend on the sum $t_1+t_2$ which with Lemma~\ref{lem:v1v2-parallel} yields the claim.
\end{proof}

\begin{corollary}
For exactly three values of $\rho$, namely, $0, 1/4, 1$  (corresponding to $X_2$, $X_5$, and $X_4$), the aspect ratio of the corresponding elliptic locus is independent of $V_1$ and $V_2$, i.e., it only depends on $(a,b)$.
\end{corollary}

\begin{proof}
For $\rho\in\{0,1/4\}$ the aspect ratio is equal to $a/b$ and for $\rho=1/4$ it is equal to $(a^2+b^2)/(2ab).$
From Appendix \ref{app:arho-brho-ratio} it follows that
$\lambda_{\rho}=a_{\rho}/b_{\rho}$ is a function of $[\rho,\cos(t_1+t_2)]=(\rho,z)$. The function $\lambda_{\rho}(\rho,z)$ is independent of $z$ when 
the level set $\partial\lambda_{\rho}/\partial{z}=0$
is a vertical straight line. Direct analysis shows that is the case exactly when $\rho\in\{0,1/4,1\}$.
\end{proof}

\begin{proposition}
The product $a_\rho\;b_\rho$ is invariant over the family of parallel $\Vs$ and given by:


\[
a_{\rho}b_{\rho}= \frac{\left|(2 \rho +1) (2 a^2 b^2 (\rho -1)+3 \rho (a^4-b^4) \cos
   (t_1+t_2)-3  \rho(a^4 + b^4)  )\right|}{18 a b} 
\]
\label{prop:product}
\end{proposition}

\begin{proof}
Obtain the above from Equation~\ref{eqn:axes-prod} in    Appendix\ref{app:arho-brho-ratio}. Since for parallel $\Vs$, $t_1+t_2$ is constant (Lemma~\ref{lem:v1v2-parallel}), the proof is complete. Alternatively, the same result could be obtained by symbolic simplification of the affine curvature of $\X$ equal to $(a_\rho b_\rho)^{-2/3}$, since it is an ellipse \cite{guggenheimer1977}.
\end{proof}

\begin{remark}
Only for $\rho\in\{0,-1/2\}$, i.e., $\X\in\{X_2,X_3\}$, is the product $a_\rho b_\rho$ independent of $\Vs$, and equal to $a b/9$ and $0$, respectively.
\end{remark}


\begin{corollary}[Axis annihilation] For each family of parallel $\Vs$, there is a unique $\rho$ (other than $-1/2$) such that the product $a_\rho b_\rho = 0$. Said $\rho$ is given by:
\[
\rho=\frac{2 a^2 b^2}{3 \left(a^4-b^4\right) \cos (t_1+t_2)+2 a^2 b^2-3 a^4-3 b^4}
\]
except when the denominator vanishes, which can only happen if $a/b<\sqrt{3}$. In this case, no $\rho$ exists. 
\end{corollary}

\begin{proof}
Follows from the expressionm for $a_\rho b_\rho$ in Proposition~\ref{prop:product}.
\end{proof}

\begin{remark}
It $a/b=\sqrt{3}$, for vertical $\Vs$, the axis can't be annihilated.
\end{remark}

\noindent Referring to Figure~\ref{fig:parallels}:

\begin{corollary}
The semi-axis lengths $a_\rho,b_\rho$ of the locus of $\X$ are invariant over the family of parallel $\Vs$.
\end{corollary}

\begin{proof}
By Proposition \ref{prop:ratio} and Lemma \ref{prop:product}, the ratio $a_\rho/b_\rho$ and product $a_\rho b_\rho$ of the axes are invariant over parallel $\Vs$, respectively. The result follows.
\end{proof}

Referring to Figure~\ref{fig:circ-loci}:

\begin{proposition}
The locus of $\X$ is a circle iff $\Vs$ is (i) horizontal with $\rho$ assuming two values $\rho_h$, or (ii) vertical with $a{\neq}3b$, and $\rho$ assuming two values $\rho_v$. These are given by:

\[ \rho_h=\frac{b}{b{\pm}3a}\, , \;\;\;\rho_v = \frac{a}{ a{\pm}3b} \]
\end{proposition}

\begin{proof}
The parametrization for $(x_\rho,y_\rho)$ in Theorem~\ref{thm:ellipse} can be developed to yield:

\begin{itemize}
\item For $t_1+t_2=\pi$ and $\rho=b/(b+3a)$, $\X$ is a circle centered in $[0,(a+b)^2\sin t_1 /(3 a+b)]$ and radius
$ a(a+b)/(b+3a)$.
\item For $t_1+t_2=\pi$ and $\rho=b/(b-3a)$, $\X$ is a circle centered in $[0,(a-b)^2\sin t_1 /(b-3 a)]$ and radius
$ a(a-b)/(3a-b)$.
\item For $t_1+t_2=0$ and $\rho=a/(a+3b)$, $\X$ is a circle centered in $[ (a+b)^2\cos t_1/(a+3b),0]$ and radius
$ b (a+b)/(a+3b)$.
\item For $t_1+t_2=0$, $\rho=a/(a-3b)$ and $a\ne 3b$, $\X$ is a circle centered in $[ (a-b)^2\cos t_1/(a-3b),0]$ and radius
$|b (a-b)/(a-3b)|$.
\end{itemize}

\begin{figure}
    \centering
    \includegraphics[width=\textwidth]{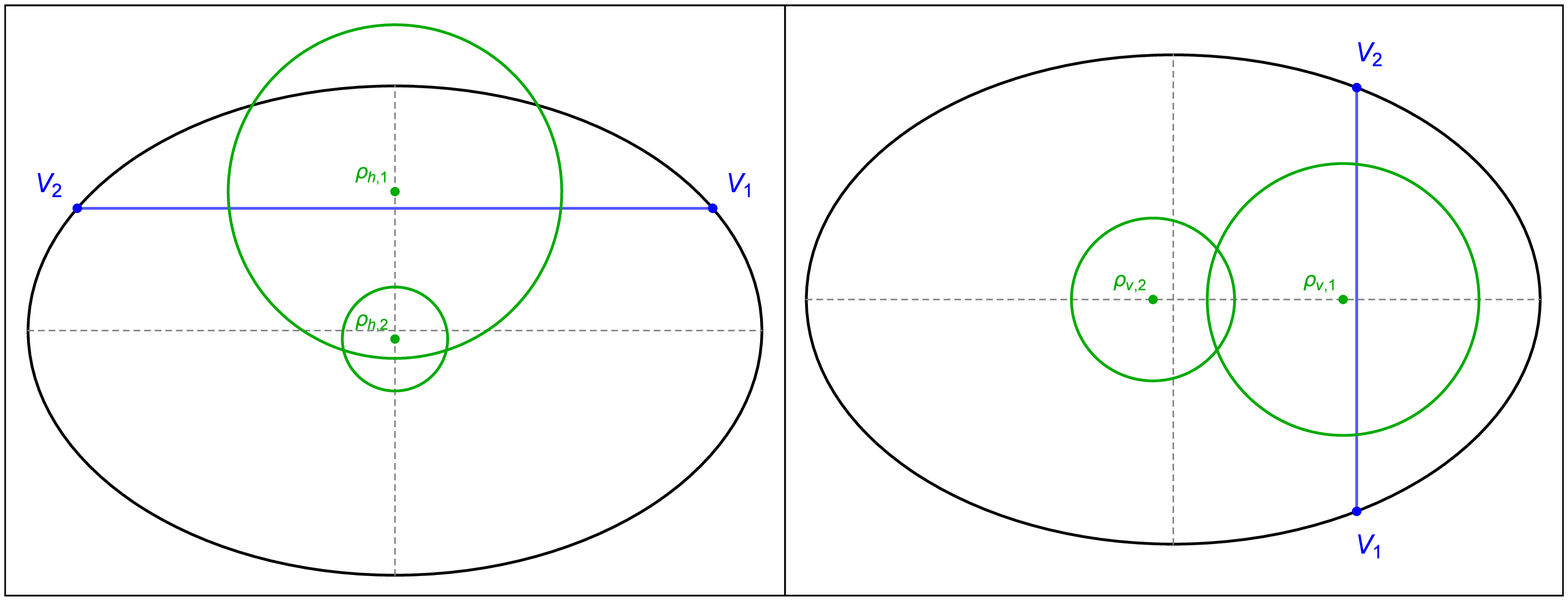}
    \caption{The loci of $\X$ (green) can only be circles when $\Vs$ (blue) are horizontal (left) or vertical (right). When horizontal (resp. vertical) only at $\rho=\{\rho_{h,1},\rho_{h,2}\}$ (resp. $\rho=\{\rho_{v,1},\rho_{v,2}\}$) is the locus a circle. As the $\Vs$ traverse all horizontals (resp. verticals), the circles will rigidly translate vertically (resp. horizontally). \href{https://youtu.be/nLeKvxcicNY}{Video}}
    \label{fig:circ-loci}
\end{figure}

If $V_1 V_2$ is neither horizontal nor vertical, there are no real solutions for $\rho$ such that $a_\rho/b_\rho=1$.

\end{proof}

\begin{proposition}
For $\Vs$ vertical, consider the case of $a=3b$ and $\rho\notin\{-1/2,3/2\}$.  The locus of $\X$ is the axis-parallel ellipse
centered at $[2b(2r+3)\cos t_1/2,0]$ and
axes $a_{\rho}= b|2r-3|/3,\;\;b_{\rho}= b|2r+1|/3$.
This ellipse is a circle when $\rho=1/2$, i.e., when $X_\rho=X_{381}$.
\end{proposition}

\begin{proof}
Direct derivation from Theorem~\ref{thm:ellipse}.
\end{proof}

\begin{remark}
By definition, $X_3$ is contained in the perpendicular bisector of $\Vs$, given by:


 \[ -2a\sin(t_1+t_2)x+2 b(\cos(t_1+t_2)+1)y+c^2(\cos t_1+\cos t_2) \sin(t_1+t_2)=0 \]

\end{remark}

\begin{proposition}
The locus of $X_3$ is a variable-length segment $P_3 P_3'$ given by:

\begin{align*}
P_{3,x}=&   \frac{ c^2 }{4a}  \left( \frac{(  1-\cos(2 t_1+2 t_2 )) }{ \sqrt{2-2\,\cos( t_1+t_2) }}+ \, \cos  t_1+\cos  t_2  \right) \\
P_{3,y}=&  - \frac{ c^2}{4b} \left(  \sqrt {
2-2\,\cos \left(t_1+t_2 \right) }+ \sin t_1 +\sin t_2  \right)  \\
P_{3,x}'=& \frac{ c^2 }{4a}  \left( \frac{(   \cos(2 t_1+2 t_2 )-1) }{ \sqrt{2-2\,\cos( t_1+t_2) }}+ \, \cos  t_1+\cos  t_2  \right) \\
P_{3,y}'=&     \frac{ c^2}{4b} \left(  \sqrt {
2-2\,\cos \left(t_1+t_2 \right) }-\sin t_1 
-\sin t_2 \right)  
\end{align*}

Furthermore, its length $L_3$ is given by:

\[ L_3=|P_3-P_3'|={\frac {c^2\sqrt {  2(d^2 -   c^2  \cos
 \left( {  t_1}+{  t_2} \right))}  }{2ab}}\]
 \label{prop:locus_X3}
 \end{proposition}

\begin{proof}
The coordinates of $X_3=[x_3,y_3]$ are given by:
\begin{align*}
x_3=&\,\frac {c^2}{4a} \left( \cos (t+t_1+t_2)+\cos{t_1}+\cos{t_2}+\cos{t} \right)  \\
 y_3=&\,\frac {c^2}{4b} \left(  \sin(t+t_1+t_2)-\sin{t_1}-\sin{t_2}-\sin{t}\right)\end{align*}

Direct calculations yield the claimed expressions.
\end{proof}
 
\begin{corollary}
The min (resp. max) of $L_3$ is $c^2/a$ (resp. $c^2/b$) and the midpoint of $P_3 P_3'$ is given by \[\left[ \frac{ c^2}{4a} \left( \cos t_1 +\cos t_2  \right)  ,- \frac{c^2 }{4b} \left( \sin t_1 +\sin t_2\right)  \right]   \]
\label{cor:midpoint_X3}
\end{corollary}
\begin{proof}
Direct from the Proposition \ref{prop:locus_X3}.
\end{proof}

\section{Locus Center Translation}
\label{sec:transl}
\noindent Referring to Figure~\ref{fig:parallels}:

\begin{proposition}
Over the family of parallel $\Vs$, the locus of $\X$ is a family of rigidly-translating ellipses whose center moves along a straight line $\L_\parallel$ passing through $O$ and given by:
\[ \L_\parallel:
 y=\frac{a}{b}\cdot\frac{(3\,{a
}^{2} +{b}^{2})\rho+2\,{b}^{2}}{({a}^{2} +3\,{b}^{2})\rho+2\,{a}^{2}}\tan\left(\frac{t_1+t_2}{2}\right)x
\]
\end{proposition}

\begin{proof}
Directly from the expression for $O_\rho$ in Theorem~\ref{thm:ellipse}.
\end{proof}

\begin{remark}
$\L_\parallel$ is perpendicular to $\Vs$ when $\rho=1$ ($X_4$).
\end{remark}

\begin{figure}
    \centering
    \includegraphics[width=\textwidth]{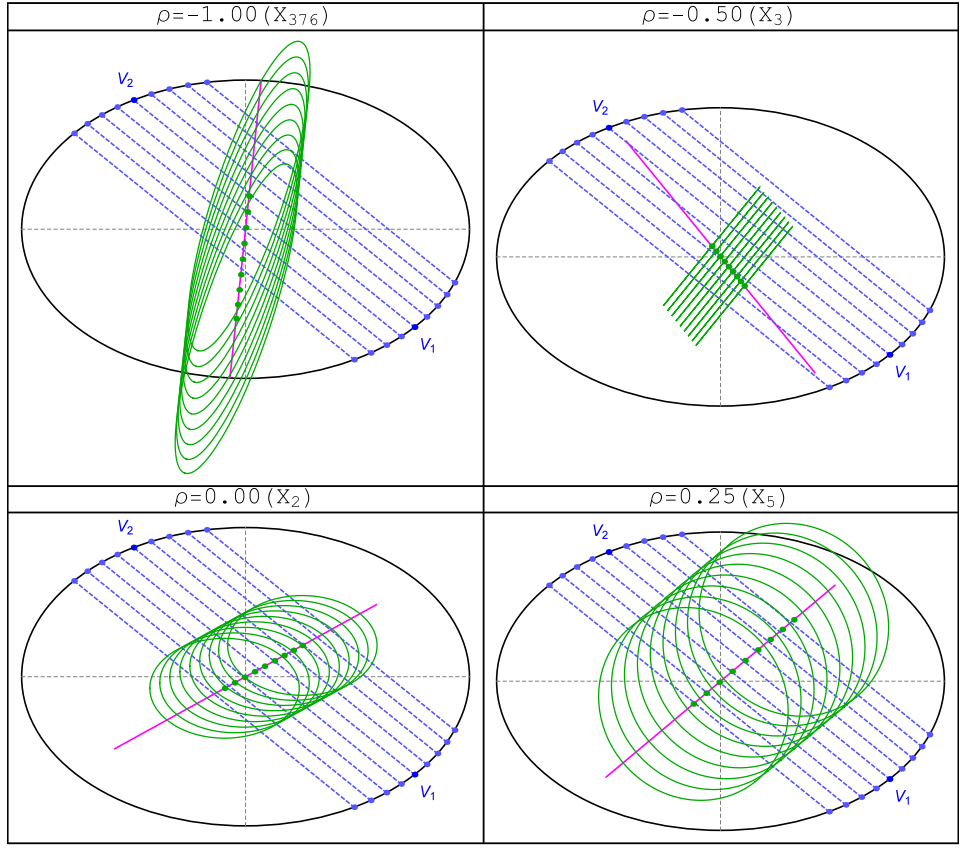}
    \caption{Over the family of parallel $\Vs$ (dashed blue, constant $t_1+t_2$), the loci of $\X$ (solid green) are a family of rigidly-translating ellipses. Their centers (green dots) move along a straight line $\L_\parallel$ (magenta) which crosses $\E$'s center. Shown are the cases for $\rho\in\{-1,-1/2,0,1/4\}$, i.e., $X_k,k=376,3,2,5$, respectively. \href{https://youtu.be/zFOeENDJRho}{Video}}
    \label{fig:parallels}
\end{figure}

\noindent Referring to Figure~\ref{fig:moving-rho}:

\begin{proposition}
For $\Vs$ stationary, as one varies $\rho$, the center $O_\rho$ of the locus of $\X$ follows a straight line $\L_\rho$ whose slope only depends on the slope of $\Vs$. In fact:
\begin{align*}
\L_\rho:& \left[ \frac {   a\left(
\cos  t_1 +\cos t_2 \right) }{3},  \frac { b (\sin  t_1+\sin t_2)    }{3
 }\right]+\\
 &+\rho \cos\left(\frac{ t_1 -t_2}{2}   
  \right)\left[\frac{{a}^{2}+3\,{b}^{2}}{3a}\cos \left(\frac{ t_1 +t_2}{2}   
  \right),\frac { 3\,{a}^{2}+{b}^{2}   }{3b } \sin \left(\frac{ t_1 +t_2}{2}
  \right)
\right]
\end{align*}
\label{prop:L_parallel}
\end{proposition}

\begin{proof}
Follows from Theorem \ref{thm:ellipse}.
\end{proof}

\begin{corollary}
The product of the slopes of $\L_\rho$ and $\Vs$ is constant over all choices of $V_1$ and $V_2$ and equal to $-\frac{3a^2+b^2}{a^2+3b^2}$.
\end{corollary}

\begin{corollary}
Only when $\Vs$ coincides with either the major or minor axis of $\E$ can $\L_\rho$ pass through the center of $\E$.
\end{corollary}

\noindent Referring to Figure~\ref{fig:parallel-center} (bottom right):

\begin{remark}
When $\Vs$ passes thru $O$, $\L_\rho$ collapses to $O$, and $O_\rho=O$ for all $\rho$.
\end{remark}

\begin{figure}
    \centering
    \includegraphics[width=.66\textwidth]{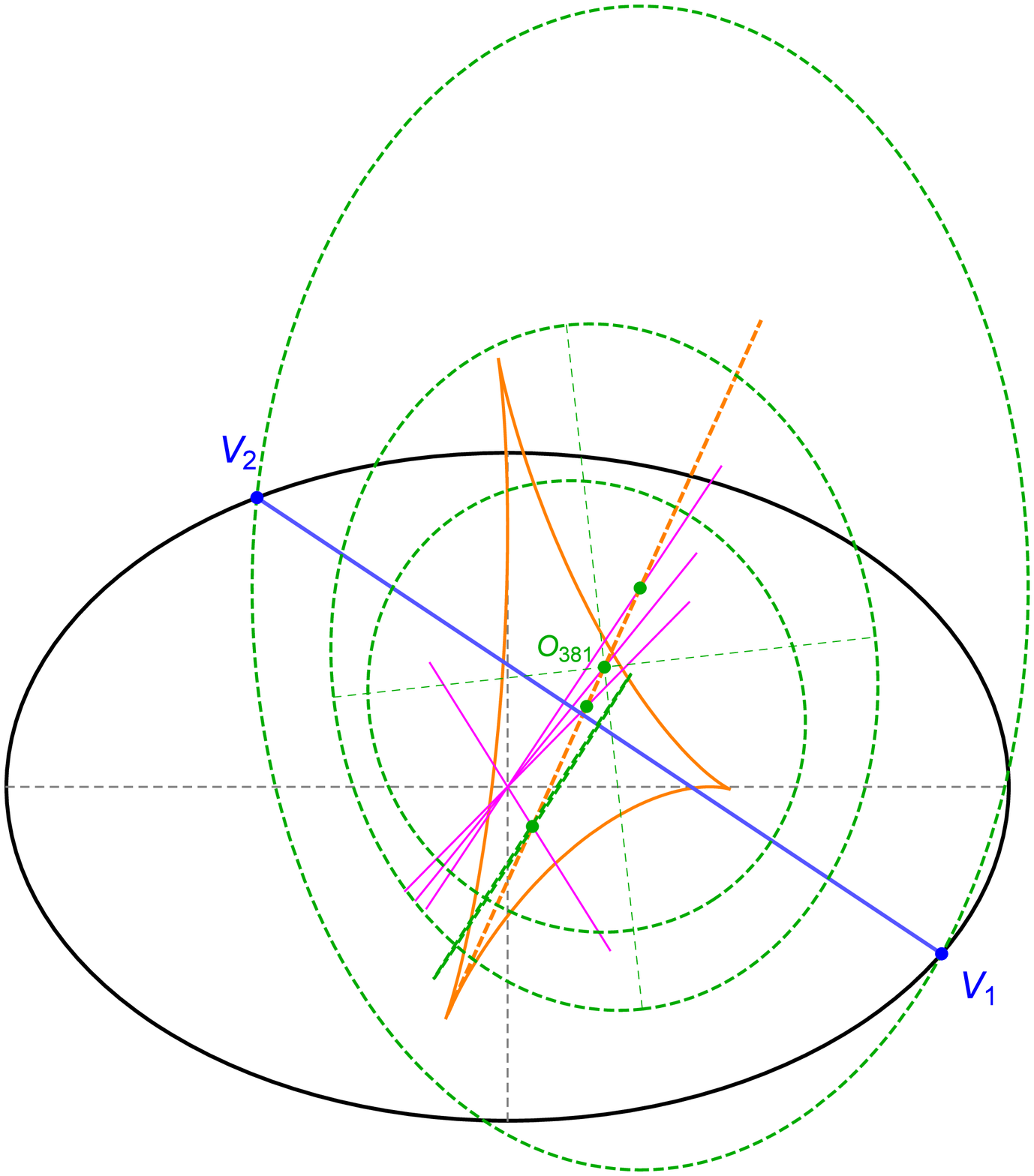}
    \caption{Elliptic loci (green ellipses) are shown for $\rho\in\{0,1/4,1/2,1\}$ with  centers $O_k,k=2,5,381,4$ on a line $\L_\rho$ (dashed orange). Also shown is a family of lines $\L_\parallel$ (magenta) passing through the center of $\E$ depicting the locus of individual $O_\rho$ as the family of parallel $\Vs$ is traversed. Also shown (solid orange) is the tricuspid envelope of $\L_\rho$ for fixed $V_1$, over all $V_2$ on $\E$. \href{https://youtu.be/w5KuN_0rQBQ}{Video}}
    \label{fig:moving-rho}
\end{figure}

\begin{figure}
    \centering
    \includegraphics[width=\textwidth]{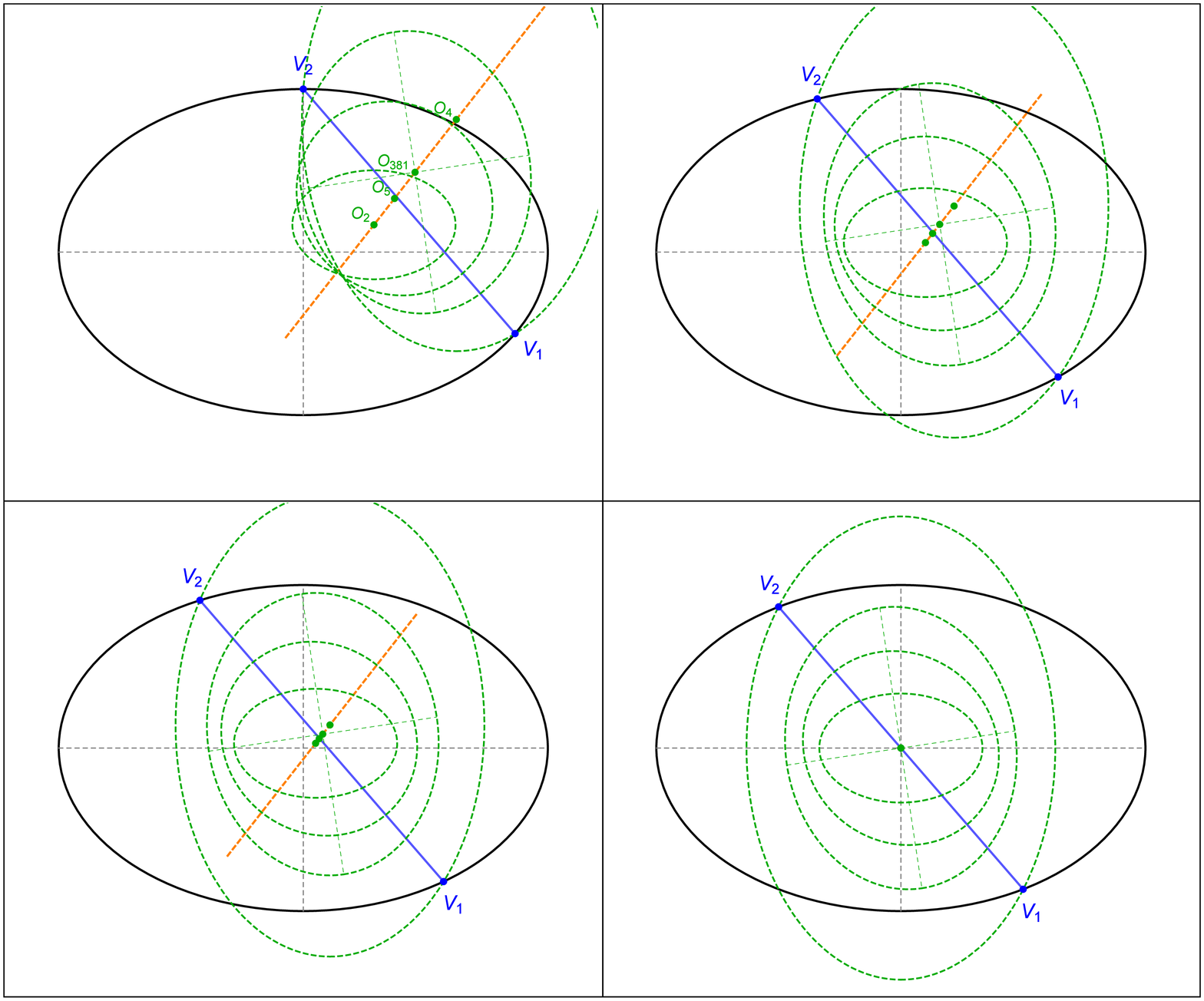}
    \caption{Four positions of parallel $\Vs$ are shown, approaching the center of $\E$. For each position, loci for $\rho=\{0,1/4,1/2,1\}$ are collinear on $\L_\rho$ (dashed orange). Notice as $\Vs$ approaches the position where it crosses the origin, centers come closer to each other and $\L_\rho$ degenerates to a point (bottom right). \href{https://youtu.be/TpBjKlkFjkg}{Video}}
    \label{fig:parallel-center}
\end{figure}

\begin{corollary}
 $O_\rho$ is on $\Vs$ at the following $\rho$: 
 \[\rho=\frac{2 a^2 b^2}{3 \left(b^4-a^4\right) \cos (t_1+t_2)+2 a^2 b^2+3 a^4+3 b^4}\]
\end{corollary}

\section{Phenomena with \torp{$V_1$}{V1} fixed and \torp{$V_2$}{V2} variable}
\label{sec:loci-env}
\subsection{Envelope of \torp{$\L_\rho$}{Lr}}

\noindent Referring to Figure~\ref{fig:moving-rho}:

\begin{proposition}
For fixed $V_1$, over all $V_2$ on $\E$, the envelope of $\L_\rho$ is a 3-cusped quartic curve $\Delta$ affine to Steiner's deltoid, given parametrically by:

\begin{align*}
\Delta_{t_1}(u)&=\left[ {\frac {\cos t_1 \left( {a}^{4}-{b}^{4}
 \right) }{2a \left( 3\,{a}^{2}+{b}^{2} \right) }},- \,{\frac {\sin
t_1 \left( {a}^{4}-{b}^{4} \right) }{2b \left( {a
}^{2}+3\,{b}^{2} \right) }}\right]\\
+&\left[{\frac { \left( 2\,\cos u +\cos \left( t_1
+2\,u\right)  \right)  \left( {a}^{4}-{b}^{4} \right) }{2a( \,3{a}^{2}+ \,{b}^{2})}},- \,{\frac { \left( 2\,\sin u-\sin \left( t_1+2\,u\right)  \right)  \left( {a
}^{4}-{b}^{4} \right) }{2b \left( {a}^{2}+3\,{b}^{2} \right) }}\right]
\end{align*}
\end{proposition}

\begin{proof} Follows from Proposition \ref{prop:L_parallel} and the definition of the envelope \cite[Chapt. 3]{guggenheimer1977}.
\end{proof}

\begin{proposition}
The area  of the region bounded by $\Delta_{t_1}$ is invariant over $t_1$ and given by
\[A(\Delta_{t_1})=\frac{\pi(a^4-b^4)^2} {ab(3a^2+b^2)(a^2+3b^2)}\]

\end{proposition}

\begin{figure}
    \centering
    \includegraphics[width=\textwidth]{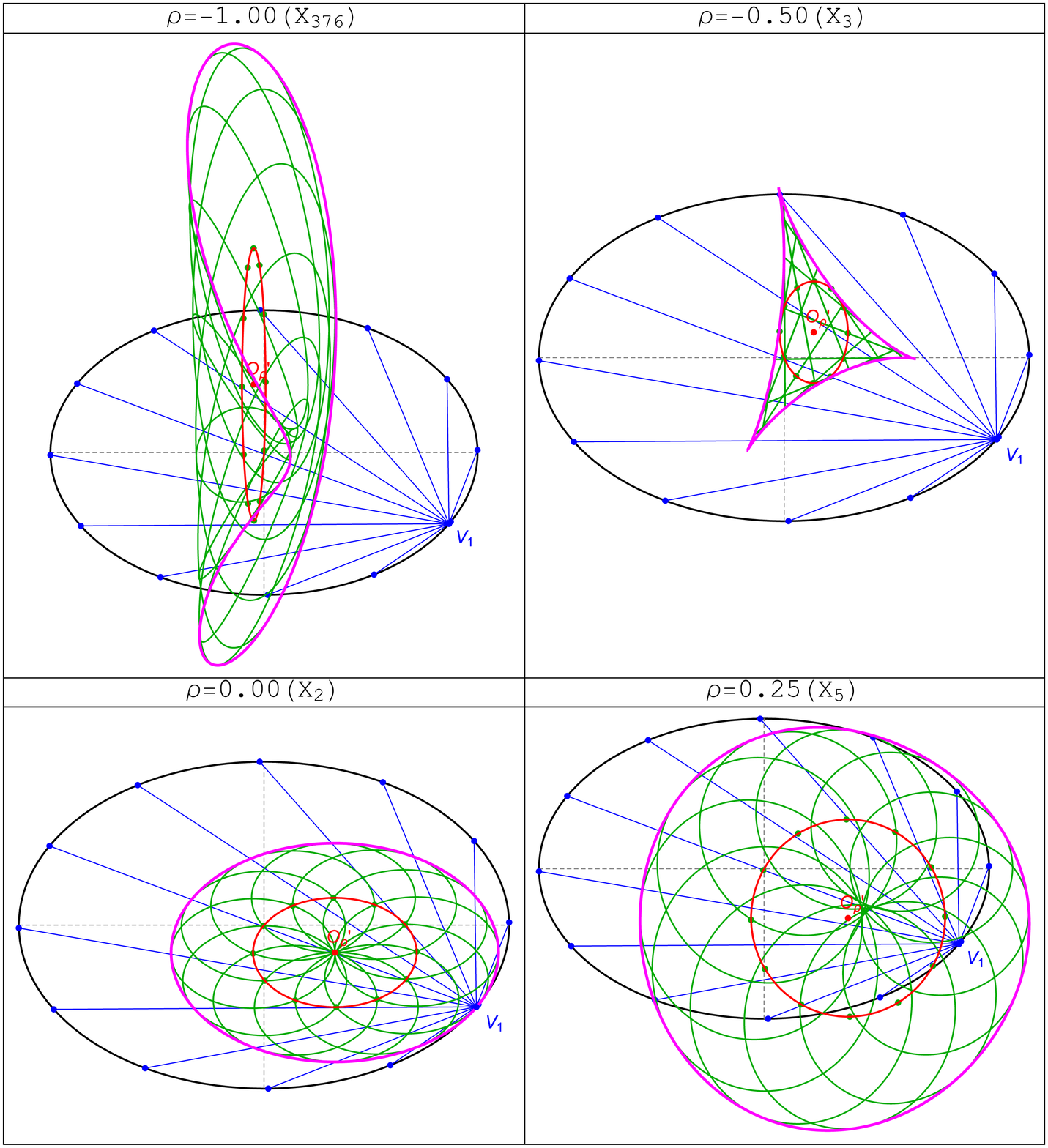}
    \caption{For a choice of $V_1$ and over $V_2$ on $\E$ (black), the loci of $\X$ are a family of ellipses (green) whose centers (green dots) sweep an ellipse $\Gamma_\rho$ passing through the center of $\E$ and centered $O_\rho'$. Shown are the cases for $\rho\in\{-1,-1/2,0,1/4\}$, i.e., $X_k,k=376,3,2,5$, respectively. Also shown are envelopes (pink) to the ellipse families. For $X_3$ (top right), the envelope is a half-sized Steiner's hat \cite{garcia2020-steiner}; for $X_2$ (bottom left) the envelope is an ellipse of fixed axes, internally tangent to $\E$ at one point. All envelopes are area-invariant wrt $V_1$.} 
    \label{fig:radial}
\end{figure}

\subsection{Locus of \torp{$O_{\rho}$}{Or}} Referring to Figure~\ref{fig:radial}:

\begin{proposition}
With $V_1$ fixed, over all $V_2$, the locus of centers $O_\rho$ of the loci of $\X$ is an ellipse $\Gamma_\rho$ centered on $O_\rho'$, which is axis-parallel with $\E$ and contains its center $O$. Its semiaxes $(a',b')$ and center are given by: 
\begin{align*}
    (a',b') =& \left(\frac{a^2 (\rho +2)+3 b^2 \rho}{6 a},\frac{3 a^2
   \rho +b^2 (\rho +2)}{6 b}\right)\\
O_\rho' = & \left[a'\cos{t_1},b'\sin{t_1}\right]
\end{align*}
\label{prop:gamma-rho}
\end{proposition}

\begin{corollary}
At $\rho=1$ ($X_4$), $\Gamma_\rho$ is an axis-parallel ellipse with aspect ratio $b/a$ with center at $\left[\frac{(a^2+b^2)\cos{t_1}}{2a},\frac{(a^2+b^2)\sin{t_1}}{2b}\right]$ and axes $(\frac{a^2+b^2}{2a},\frac{a^2+b^2}{2b})$.
\end{corollary}

\begin{corollary}
At $\rho=0$ ($X_2$), $\Gamma_\rho$ is an ellipse with aspect ratio $a/b$ centered at $\left[\frac{1}{3} a \cos {t_1},\frac{1}{3} b \sin
   {t_1}\right]$ with axes $(\frac{a}{3},\frac{b}{3})$.
\end{corollary}

\begin{corollary}
Over all $V_1$ the locus of $O_\rho'$ is an ellipse $\Gamma_\rho'$ whichi s axis-parallel and concentric with $\E$. The semi-axes of $\Gamma_\rho'$ are also $(a',b')$.
\end{corollary}

\noindent Referring to Figure~\ref{fig:tang-env}:

\begin{remark}
If $V_1$ is fixed at the left (resp. top) vertex of $\E$, over all $V_2$, $\Gamma_\rho$ is axis-parallel with $\E$ and tangent at $O$ to its minor (resp. major) axis.
\end{remark}

\begin{figure}
    \centering
    \includegraphics[width=.8\textwidth]{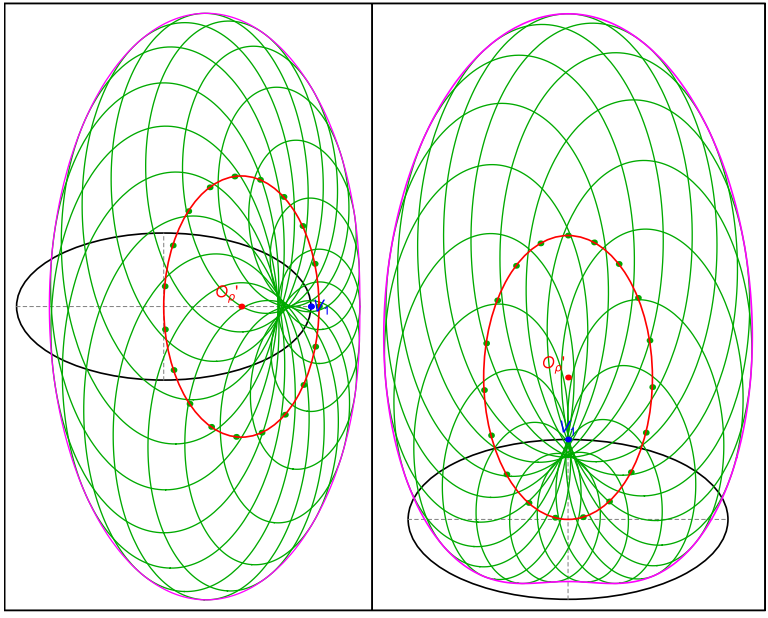}
    \caption{Locus $\Gamma_\rho$ (red) of the centers of $\X$ (green) of $\X$ for $\rho=2/3$, when $V_1$ coincides with the left (resp. top) vertex of $\E$ (black). Notice $\Gamma_\rho$ (red) is tangent at the center of $O$ to the minor (resp. major) axis of $\E$ and its center $O_\rho'$ lies on the major (resp. minor) axis of $\E$.}
    \label{fig:tang-env}
\end{figure}

\subsection{Envelope of the family of elliptic \torp{$\X$}{Xr}}

\begin{proposition}
With $V_1$ stationary and $V_2$ sweeping the boundary of $\E$,  a regular part of the envelope of $\X$ is a  curve $\Gamma_{t_1}$ parametrized by

\begin{align*}
  x_{t_1}&=    \frac { \left[  ( {a}^{2}+3\,{b}^{2})  ( 2\,\cos t+\cos t_1 ) -3\,c^2\cos
 (t_1+2\,t )  
 \right] \rho}{6a} +\frac{a}{3} \left( 2\,\cos t+\cos t_1 \right) \\
  y_{t_1}&=   \frac{   \left[ ( 3\,{a}^{2}+{b}^{2}  ) ( 2\,\sin t+\sin t_1) -3\,c^2 \sin( t_1+2\,t  )         \right] \rho}{6b} +\frac{b}{3} \left( 2\,\sin t+\sin t_1 \right) 
\end{align*}
\end{proposition}

\begin{proof}
Direct from the definition of an envelope \cite{guggenheimer1977} via CAS simplification.
\end{proof}

\begin{proposition}
The area (algebraic) of the region bounded by $\Gamma_{t_1}$ is invariant over $t_1$ and given by
\[A(\Gamma_{t_1})= \frac{\pi}{9} \left[ \frac {\ \left( 15\,{a}^{4}+2\,{a}^{2}{b}^{2}+15\,{b}^{4}
 \right) {\rho}^{2}}{2 a b}+ {\frac {2 \left( 3\,{a}^{4}+2\,{a}^{2}
{b}^{2}+3\,{b}^{4} \right) \rho}{a b}}+  {4 ab}\right]
\]

\end{proposition}

\begin{proof} Follows by direct integration of $A(\Gamma_{t_1})=\frac{1}{2} \int_{\Gamma_{t_1}}(xdy-ydx)$.
\end{proof}

In \cite{garcia2020-steiner} we called {\em Steiner's Hat} the negative pedal curve of an ellipse with respect to a point $M$ on the boundary. One of its curious properties is that it is area-invariant over all $M$.

\begin{remark}
With $V_1$ stationary and $V_2$ sweeping the boundary of $\E$, the envelope of the locus of $X_3$ over all positions of $V_2$ is 2:1 homothetic to Steiner's Hat.
\end{remark}

\begin{figure}
    \centering
    \includegraphics[width=.8\textwidth]{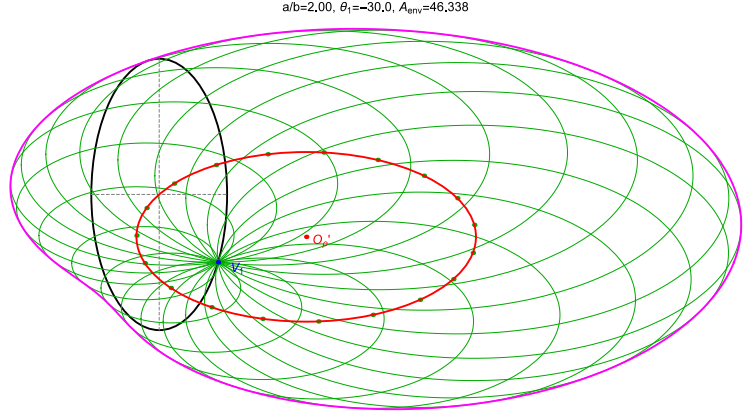}
    \caption{Let $V_1$ be a point on $\E$ (black, rotated 90 degrees to save space, $a/b=2$). Over $V_2$, the loci of $X_4$ are ellipses (green) which all pass through $V_1$. Their centers (green dots) sweep an axis-parallel ellipse (red) centered at $O_\rho'$, whose aspect ratio is $b/a$. The exterior envelope of said loci is (in general) a non-convex curve (pink) which is the affine image of Pascal's Limaçon. Its area is independent of $V_1$, and it is tangent to $\E$ at at least one point. \href{https://youtu.be/sPQrz7ddRfA}{Video}}
    \label{fig:radial-x4}
\end{figure}

This stems from the fact that it cuts $\Vs$ perpendicularly and at its midpoint.

\noindent Referring to Figure~\ref{fig:radial} (bottom left) and Figure~\ref{fig:radial-x4}:


\begin{remark}
For $\rho=0$ (resp. $\rho=1$), the external envelope of $\X$ is internally (resp. externally) tangent to $\E$ at $V_1$ (resp. at the point(s) on $\E$ whose normal goes through $V_1$). For $\rho=0$ the envelope is elliptic with axes $2a/3$ and $2b/3.$
\end{remark}

\begin{proposition}
For $\rho=1$, the external envelope of $\X$ is the affine image of Pascal's Limaçon.  
\end{proposition}

\begin{proof}
A construction for Pascal's Limaçon is given in \cite[Pascal's Limaçon]{mw} as follows: specify a fixed point $P$ and a circle $C$. Then draw all circles with centers on $C$ which pass through $P$. The external envelope of said circles is the Limaçon. Referring to Figure~\ref{fig:radial-x4}, apply an affine transformation that sends $\Gamma_\rho$ to a circle $C$. This will automatically send all $\X$ ellipses to (variable radius) circles, since they have the same aspect ratio and are axis-parallel with $\Gamma_\rho$ (Proposition~\ref{prop:gamma-rho}. Therefore, the affine image of the $\X$ becomes a family of circles with centers on $C$ through a common point $P$, the affine image of $V_1$.
\end{proof}

\section{Conclusion}
\label{sec:conclusion}
This article studied properties of the loci of triangle centers over of a special of ellipse-inscribed triangles. The following questions are still unanswered:

\begin{itemize}
    \item Is there a triangle center which is not a fixed affine combination of $X_2,X_4$ whose locus over $\T(t)$ is an ellipse? We did not find one amongst all 38k+ centers listed in \cite{etc}.
    \item For $V_1$ fixed and over all $V_2$ on $\E$, are there interesting properties of the internal envelope of the family of $\X$ ellipses? For $X_2,X_4$, this envelope is a point $(O+(V_1-O)/3)$, and $V_1$, respectively. However for other $\rho$ this envelope is more complex.
\end{itemize}

Animations illustrating the dynamic geometry of some of the above phenomena appear on Table~\ref{tab:videos}. 

\begin{table}[H]
\begin{tabular}{|c|l|c|}
\hline
Id & Title & \texttt{youtu.be/...}  \\
\hline
01 & Basic elliptic loci & \href{https://youtu.be/zjiNgfndBWg}{\texttt{zjiNgfndBWg}} \\
02 & $\X$ slides on Euler line & \href{https://youtu.be/w5KuN_0rQBQ}{\texttt{w5KuN\_0rQBQ}} \\
03 & Locus of $\X$ over parallel $\Vs$ & \href{https://youtu.be/zFOeENDJRho}{\texttt{zFOeENDJRho}} \\
04 & \makecell[lt]{Relative motion of loci over parallel $\Vs$} & \href{https://youtu.be/TpBjKlkFjkg}{\texttt{TpBjKlkFjkg}} \\
05 & \makecell[lt]{Circular loci if $\Vs$ are horiz. or vert.} & \href{https://youtu.be/nLeKvxcicNY}{\texttt{nLeKvxcicNY}} \\
06 & \makecell[lt]{Limaçon-Like envelopes of $X_4$ locus family} & \href{https://youtu.be/sPQrz7ddRfA}{\texttt{sPQrz7ddRfA}} \\
\hline
\end{tabular}
\caption{Illustrative animations, click on the link to view it on {YouTube} and/or enter \texttt{youtu.be/<code>} as a URL in your browser, where \texttt{<code>} is the provided string.}
\label{tab:videos}
\end{table}

\noindent We are very grateful to A. Akopyan and P. Moses for key insights. We thank the journal referees for their valuable suggestions. We thank P.N. de Souza for his crucial editorial help. The second author is fellow of CNPq and coordinator of Project PRONEX/ CNPq/ FAPEG 2017 10 26 7000 508.

\appendix

\section{Axis ratio of locus of \torp{$\X$}{Xr}}
\label{app:arho-brho-ratio}
Here we assume the origin is at the center of the $\X$ locus. Then the latter can be expressed as:

\[\X:\,a_{20}x_1^2+2a_{11}x_1 y_1+a_{02}y_1^2 +a_{00}=0\]

The aspect ratio is given by:

\[ \frac{a_{\rho}}{b_{\rho} } = \frac{a_{20}+a_{02}+
\sqrt{(a_{20}-a_{02})^2+4a_{11}^2 } }{2(a_{20}a_{02}-a_{11}^2)}\]

The product of axes is given by:

\begin{equation}
a_{\rho}b_{\rho}=\frac{|a_{00}|}{\sqrt{a_{20}a_{02}-a_{11}^2}}
\label{eqn:axes-prod}
\end{equation}
where (recall $z=\cos(t_1+t_2)$):
{\small
\begin{align*}
    a_{20}&= 54\,{a}^{2}\rho c^2    \left( 3\,{a}^{2}
\rho+{b}^{2}\rho+2\,{b}^{2} \right) z \\
&-18\,{a}^{2} \left( 9\,{a}^{4}-6\,{b}^{2}{a}^{2}+5\,{b}^{4} \right) {\rho}
^{2}-36\,{a}^{2}{b}^{2} \left( 3\,{a}^{2}+{b}^{2} \right) \rho-36\,{a}^{2
}{b}^{4}
 \\
     a_{11}& = 54\rho (\rho-1) a b c^4 \sqrt{1-z^2} \\
      a_{02}&= 54\,{b}^{2}\rho c^2   \left( {a}^{2}\rho+3
\,{b}^{2}\rho+2\,{a}^{2} \right) z \\
&- 18\,{b}^{2} \left( 5\,{a}^{4}-6\,{b}^{2}{a}^{2}+9\,{b}^{4} \right) {\rho}
^{2}-36\,{b}^{2} \left( {a}^{2}+3\,{b}^{2} \right) {a}^{2}\rho-36\,{b}^{2
}{a}^{4}
 \\
a_{00}&= \left( 2\,\rho+1 \right) ^{2} \left( 3\,{a}^{4}\rho z-3\,{b}^{4}\rho z-3\,{a}^{4
}\rho+2\,{a}^{2}{b}^{2}\rho-3\,{b}^{4}\rho-2\,{b}^{2}{a}^{2} \right) ^{2}
\end{align*}
}



\section{Triangle Centers at fixed \torp{$\rho$}{r}}
\label{app:fixed-rho}
Amongst the 4.9k triangle centers on the Euler line \cite{etc-central-lines}, only the following 226 are fixed affine combinations of $X_2$ and $X_4$: $X_k,k=${\small 2, 3, 4, 5, 20, 140, 376, 381, 382, 546, 547, 548, 549, 550, 631, 
632, 1564, 1656, 1657, 2041, 2042, 2043, 2044, 2045, 2046, 2675, 
2676, 3090, 3091, 3146, 3522, 3523, 3524, 3525, 3526, 3528, 3529, 
3530, 3533, 3534, 3543, 3545, 3627, 3628, 3830, 3832, 3839, 3843, 
3845, 3850, 3851, 3853, 3854, 3855, 3856, 3857, 3858, 3859, 3860, 
3861, 5054, 5055, 5056, 5059, 5066, 5067, 5068, 5070, 5071, 5072, 
5073, 5076, 5079, 7486, 8703, 10109, 10124, 10299, 10303, 10304, 
11001, 11539, 11540, 11541, 11737, 11812, 12100, 12101, 12102, 12103, 
12108, 12811, 12812, 14093, 14269, 14782, 14783, 14784, 14785, 14813, 
14814, 14869, 14890, 14891, 14892, 14893, 15022, 15640, 15681, 15682, 
15683, 15684, 15685, 15686, 15687, 15688, 15689, 15690, 15691, 15692, 
15693, 15694, 15695, 15696, 15697, 15698, 15699, 15700, 15701, 15702, 
15703, 15704, 15705, 15706, 15707, 15708, 15709, 15710, 15711, 15712, 
15713, 15714, 15715, 15716, 15717, 15718, 15719, 15720, 15721, 15722, 
15723, 15759, 15764, 15765, 16239, 16249, 16250, 16446, 17504, 17538, 
17578, 17800, 18585, 18586, 18587, 19708, 19709, 19710, 19711, 21734, 
21735, 23046, 33699, 33703, 33923, 34200, 34551, 34552, 34559, 34562, 
35018, 35381, 35382, 35384, 35400, 35401, 35402, 35403, 35404, 35405, 
35406, 35407, 35408, 35409, 35410, 35411, 35412, 35413, 35414, 35415, 
35416, 35417, 35418, 35419, 35420, 35421, 35434, 35435, 35732, 35734, 
35735, 35736, 35737, 35738, 36436, 36437, 36438, 36439, 36445, 36448, 
36454, 36455, 36456, 36457, 36463, 36466}.

\section{Focus-Mounted Triangles}
\label{app:focus-mounted}
Consider the 1d family of triangles $\T_f(t)=f_1 f_2 P(t)$ where $P(t)$ sweeps $\E$ and the $f_j$ are the foci of $\E$, at $[{\pm}c,0]$. Referring to Figure  ~\ref{fig:focus-mounted}, noting that the result for $X_1,X_2$ were proved in \cite[Thm 2.2.1]{dykstra2006-loci}:

\begin{proposition}
Over the first 1000 triangle centers listed in \cite{etc}, only the loci of $X_k$, $k=1,2,8,10,145,551$ over $\T_f(t)$ are ellipses. The first four are given by:
{\small
\begin{align*}
X_1:& \;\; \frac{x^{2}}{c^{2}} +\frac { \left( a+c \right) ^{2}{y}^{2}}{b^{2}c^{2}}-1=0
  \\
 X_2:&\;\; \frac{9x^2}{a^2}+\frac{9y^2}{b^2}-1=0\\
  X_8:&\;\;   {\frac {x^{2}}{ \left( a-2\,c \right) ^{2}}}+{\frac {y^{2}
 \left( a+c \right) ^{2}}{b^{2} \left( a-c \right) ^{2}}}-1=0
             \\
   X_{10}:&\;\;   \frac {4x^{2}}{ \left( a-c \right) ^{2}}+ \frac { 4\left( a+
c \right) ^{2}y^{2}}{a^{2}b^{2}}-1=0
\end{align*}
}
\label{prop:focal}
\end{proposition}

\begin{remark}
The vertices of the locus of $X_1$ are $f_1,f_2$.
\end{remark}

\begin{remark}
At $a=2c$, i.e., $a/b=2/\sqrt{3}{\simeq}1.1547$, the locus of $X_8$ is the vertical segment $[0,{\pm}b/3]$. At this aspect ratio, when $P(t)$ is at the top or bottom vertex of $\E$, $\T_f(t)$ is equilateral.
\end{remark}

Recall Line $X_1 X2$ is known as the {\em Nagel Line} \cite{etc-central-lines}. For the entire 40k+ centers in \cite{etc}, the following 62 are fixed affine combinations of $X_1,X_2$ (boldface indicates those in Proposition~\ref{prop:focal}): {\small \textbf{1}, \textbf{2}, \textbf{8}, \textbf{10}, \textbf{145}, \textbf{551}, 1125, 1698, 3241, 3244, 3616, 3617, 3621, 
3622, 3623, 3624, 3625, 3626, 3632, 3633, 3634, 3635, 3636, 3679, 
3828, 4668, 4669, 4677, 4678, 4691, 4701, 4745, 4746, 4816, 5550, 
9780, 15808, 19862, 19872, 19875, 19876, 19877, 19878, 19883, 20014,
20049, 20050, 20052, 20053, 20054, 20057, 22266, 25055, 31145, 31253,
34595, 34641, 34747, 36440, 36444, 36458, 36462}.

Note that $X_2$ is the anticomplement of itself. Note also that $X_{10}$ (resp. $X_{145}$) is the anticomplement of $X_1$ (resp. $X_8$). Numerically analyzing the loci of all 38k+ on \cite{etc} which are not on the $X_1 X_2$ line we found none which produced an ellipse over $\T_f(t)$. In turn this leads to the following conjecture:



\begin{conjecture}
The locus of $X_k$ over $\T_f(t)=f_1 f_2 P(t)$ is an ellipse iff $X_k$ is a fixed affine combination of $X_1$ and $X_2$.
\end{conjecture}

\begin{figure}
    \centering
    \includegraphics[width=.8\textwidth]{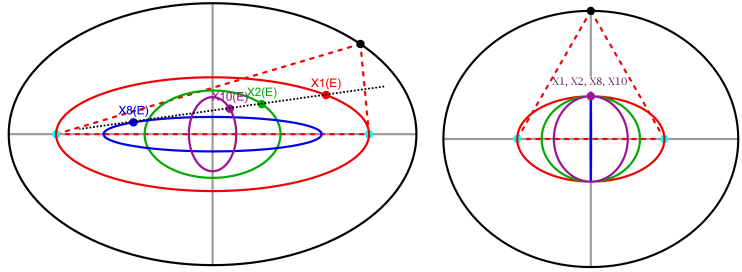}
    \caption{\textbf{Left:} elliptic loci of $X_k,k=1,2,8,10$ (collinear on the Nagel line, dashed black) for a triangle family (dashed red) with two vertices on the foci.  \href{https://bit.ly/3550ZXg}{animation 1} \textbf{Right:} at $a/b=2/\sqrt{3}{\simeq}1.1547$, the family contains two equilaterals (one shown dashed red), and therefore all loci are tangent at two point. Also at this $a/b$, the locus of $X_8$ (blue) is a vertical segment.
    \href{https://bit.ly/2HdjRLp}{animation 2}}
    \label{fig:focus-mounted}
\end{figure}

\section{Table of Symbols}
\label{app:symbols}
Symbols used in the article appear on Table~\ref{tab:symbols}.

\begin{table}[H]
{\small
\begin{tabular}{|c|l|}
\hline
symbol & meaning \\
\hline
$\E,a,b$ & base ellipse and its semi-axes \\
$O,f_1,f_2$ & center of foci of $\E$ \\
$V_1,V_2$ & points fixed on $\E$ at $t_1,t_2$ \\
$P(t)$ & moving 3rd vertex of $\T$ \\
$\T(t)$ & $\E$-inscribed triangle $V_1 V_2 P(t)$ \\
$\L_e$ & Euler line $X_2 X_4$ \\
$\X,\rho$ & $\X=X_2+\rho(X_4-X_2)$ \\
\hline
$c^2,d^2$ & $a^2-b^2$, and $a^2+b^2$, resp. \\
$z$ & shorthand for $\cos(t_1+t_2)$ \\
\hline
$O_\rho$ & center of elliptic locus of $\X$ \\
$\L_\rho,\L_\parallel$ & linear locus of $O_\rho$ over $\rho$ (resp. $\Vs$ parallels) \\
$\Delta_{t_1}$ & envelope of $\L_\rho$ for fixed $V_1$ over $V_2$ on $\E$ \\
$\Gamma_\rho$ & elliptic locus of $O_\rho$ for fixed $V_1$ over $V_2$ on $\E$ \\
$O_\rho',\Gamma_\rho'$ & center of $\Gamma_\rho$ and its elliptic locus over all $V_1$ on $\E$ \\

\hline
$X_1,X_2$ & incenter, barycenter \\
$X_3,X_4$ & circumcenter, orthocenter \\
$X_5,X_8$ & 9-pt circle center and Nagel point \\
$X_{10},X_{381}$ & Spieker center and $X_2 X_4$ midpoint \\
\hline
\end{tabular}
}
\caption{Symbols used in the article.}
\label{tab:symbols}
\end{table}

\bibliographystyle{maa}
\bibliography{references,authors_rgk_v3}

\end{document}